\documentclass[mathematics,article,accept,moreauthors,pdftex,12pt,a4paper]{mdpi} 
\lastpage{614}
\doinum{10.3390/math3030604}
\pubvolume{3}
\pubyear{2015}
\history{Received: 7 April 2015 / Accepted: 26 June 2015 / Published: 3 July 2015}

\usepackage{graphicx}
\usepackage{amsfonts,latexsym,enumitem}

\theoremstyle{mdpi}
\newcounter{thm}
\newcounter{ex}
\newcounter{re}
\newtheorem{Theorem}[thm]{Theorem}
\newtheorem{Proposition}[thm]{Proposition}

\newtheorem{Problem}[thm]{Problem}

\newtheorem{Remark}[re]{Remark}

\newtheorem{Definition}[thm]{Definition}

\Title{Topological Integer Additive Set-Sequential Graphs}

\Author{Sudev Naduvath $^{1,}$*, Germina Augustine $^{2}$ and Chithra Sudev $^{3}$}

\address{%
$^{1}$ Department of Mathematics, Vidya Academy of Science \& Technology, Thrissur - 680501, Kerala, India.\\ 
$^{2}$ PG \& Research Department of Mathematics, Mary Matha Arts \& Science College,  Mananthavady- 670645, Kerala, India.\\
$^{3}$ Naduvath Mana, Nandikkara, Thrissur-680301, India.}

\corres{e-mail:sudevnk@gmail.com, telephone: +919497557876.}

\abstract{Let $\mathbb{N}_0$ denote the set of all non-negative integers and $X$ be any non-empty subset of $\mathbb{N}_0$. Denote the power set of $X$ by $\mathcal{P}(X)$. An integer additive set-labeling (IASL) of a graph $G$ is an injective set-valued function $f:V(G)\to \mathcal{P}(X)$ such that the induced function $f^+:E(G) \to \mathcal{P}(X)$ is defined by $f^+ (uv) = f(u)+ f(v)$, where $f(u)+f(v)$ is the sumset of $f(u)$ and $f(v)$. If the associated set-valued edge function $f^+$ is also injective, then such an IASL is called an integer additive set-indexer (IASI). An IASL $f$ is said to be a topological IASL (TIASL) if $f(V(G))\cup \{\emptyset\}$ is a topology of the ground set $X$. An IASL is said to be an integer additive set-sequential labeling (IASSL) if $f(V(G))\cup f^+(E(G))= \mathcal{P}(X)-\{\emptyset\}$. An IASL of a given graph $G$ is said to be a topological integer additive set-sequential labeling of $G$, if it is a topological integer additive set-labeling as well as an integer additive set-sequential labeling of $G$. In this paper, we study the conditions required for a graph $G$ to admit this type of IASL and propose some important characteristics of the graphs which admit this type of IASLs. }



\begin{document}


\section{Introduction}

For all terms and definitions, other than newly defined or specifically mentioned in this paper, we refer to \cite{BM} and \cite{FH} and for different graph classes, we further refer to \cite{BLS} and \cite{JAG}. For the terms and concepts in Topology, we further refer to \cite{KDJ1}, \cite{JRM} and \cite{VK}. Unless mentioned otherwise, the graphs considered in this paper are simple, finite, non-trivial and connected.

Let $A$ and $B$ be two non-empty sets. The {\em sumset} of $A$ and $B$ is denoted by $A+B$ and is defined by $A+B=\{a+b: a\in A, b\in B\}$. If a set $C=A+B$, then $A$ and $B$ are said to the summands of $C$. Using the concepts of sumsets of two sets, the following notion has been introduced.

Let $\mathbb{N}_0$ denote the set of all non-negative integers. Let $X$ be a non-empty subset of $\mathbb{N}_0$ and $\mathcal{P}(X)$ be its power set. An {\em integer additive set-labeling} (IASL) of a graph $G$ is an injective function $f : V (G) \to P(X)$ such that the image of the induced function $f^+: E(G) \to \mathcal{P}(\mathbb{N}_0)$, defined by $f^+(uv)=f(u)+f(v)$ is contained in $\mathcal{P}(X)$, where $f(u) + f(v)$ is the sumset of $f(u)$ and $f(v)$ (see \cite{GA},\cite{GS1}). It is to be noted that for an IASL $f$, we have $f^+(E(G))\subseteq \mathcal{P}(X)$.  A graph that admits an IASL is called an {\em integer additive set-labeled graph} (IASL-graph). An IASL $f$ of a given graph $G$ is said to be an {\em integer additive set-indexer} (IASI) if the associated function $f^+$ is also injective. 

If $X=\mathbb{N}_0$, then it can be easily verified that every graph admits an integer additive set-labeling and an integer additive set-indexer.

The cardinality of the set-label of an element (a vertex or an edge) of a graph $G$ is said to be the {\em set-indexing number} of that element. Since the set-label of every edge of $G$ is the sumset of the set-labels of its end vertices, we do not permit labeling vertices and hence the edges of an IASL-graph $G$ with the empty set. If any of the given two sets is countably infinite, then their sumset is also a countably infinite set. Hence, all sets we consider in this paper are non-empty finite sets of non-negative integers. More studies on different types of integer additive set-labeled graphs have been done in \cite{GS0,GS2,GS11,GS12}.

An IASL $f$ of a graph $G$, with respect to a ground set $X\subset \mathbb{N}_0$, is said to be a {\em topological IASL} (TIASL) of $G$ if $\mathcal{T}=f(V(G))\cup \{\emptyset\}$ is a topology on $X$ (see \cite{GS13}). Certain characteristics and structural properties of TIASL-graphs have been studied in \cite{GS13}.

The following are the major results on topological IASL-graphs obtained in \cite{GS13}.

\begin{Proposition}\label{P-TIASI1}
	{\rm \cite{GS13}} If $f:V(G)\to \mathcal{P}(X)-\{\emptyset\}$ is a TIASL of a graph $G$, then $G$ must have at least one pendant vertex.
\end{Proposition}

\begin{Proposition}\label{P-TIASI2}
	{\rm \cite{GS13}} Let $f:V(G)\to \mathcal{P}(X)-\{\emptyset\}$ be a TIASL of a graph $G$. Then, the vertices whose set-labels contain the maximal element of the ground set $X$ are pendant vertices, adjacent to the vertex having the set-label $\{0\}$.
\end{Proposition}

If $0 \in X$, then every subset $A_i$ of $X$ can be written as a sumset as $A_i=\{0\}+A_i$. This sumset representation is known as a trivial sumset representation of $A_i$ and the summands $\{0\}$ and $A_i$ of the set $A_i$ are called trivial summands of $A_i$ in $X$.

Motivated by the studies of set-sequential graphs made in \cite{AGKS} and \cite{AH} and the studies of topological set-sequential graphs made in \cite{AGPR}, we extend our studies on integer additive set-sequential graphs, studied in \cite{GS15}, for which the collection of set-labels of whose elements forms a topology of the ground set $X$.

\section{Topological IASS-Graphs}

Let us first recall the definition of an integer additive set-sequential labeling of a given graph $G$. 

\begin{Definition}{\rm 
	{\rm \cite{GS15}} An integer additive set-labeling $f$ of a graph $G$, with respect to a finite set $X$, is said to be an {\em integer additive set-sequential labeling} (IASSL) of $G$ if there exists an extension function $f^{\ast}:G\to  \mathcal{P}(X)$ of $f$ defined by 
		\begin{equation*}
			f^{\ast}(G)=
			\begin{cases}
				f(x) & \mbox{if} ~ x\in V(G)\\
				f^+(x) & \mbox{if} ~ x\in E(G) 
			\end{cases}
		\end{equation*}
such that $f^*(G)\cup\{\emptyset\}=\mathcal{P}(X)$ .
		
A graph which admits an integer additive set-sequential labeling is called an \textit{integer additive set-sequential} graph (IASS-graph). }
\end{Definition}

We shall now introduce the notion of a topological integer additive set-sequential graphs as follows.

\begin{Definition}{\rm 
	{\rm \cite{GS15}} An integer additive set-labeling $f$ of a graph $G$, with respect to a finite set $X$, is said to be a {\em topological integer additive set-sequential labeling} (TIASSL) if  $f(V)\cup \{\emptyset\}$ is a topology on $X$ and $f^*(G)=\mathcal{P}(X)-\{\emptyset\}$, where $f^*(G)=f(V(G))\cup f^+(E(G))$, as defined in the previous definition. }
\end{Definition}

In other words, a TIASSL of a graph $G$ is an IASI of $G$ which is both a TIASL as well as an IASSL of $G$.

A topological integer additive set-sequential labeling $f:V(G)\to \mathcal{P}(X)$ is said to be a {\em topological integer additive set-sequential indexer} (TIASSI) if the induced function $f^*$ is also injective.

\begin{Remark}\label{R-ITIASL1}{\rm
	Let $f$ be a TIASSL of a given connected graph $G$, with respect to a given ground set $X$. Then, since $f(V(G))\cup \{\emptyset\}$ is a topology of $X$, $X \in f(V)$. Therefore, $\{0\} \in f(V)$ and the vertex having the set-label $X$ can be adjacent only to the vertex having $\{0\}$. Then, the vertex having the set label $X$ is a pendant vertex of $G$.} 
\end{Remark}

An IASL $f:V(G) \to \mathcal{P}(X)$, with respect to a finite ground set $X$, is said to be an  {\em integer additive set-graceful labeling} if the induced edge-function $f^+: E(G)\to \mathcal{P}(X)$ satisfies the condition $f^+(E(G))= \mathcal{P}(X)-\{\emptyset, \{0\}\}$ (see \cite{GS14}).

A {\em topological integer additive set-graceful labeling} (Top-IASGL) with respect to a finite ground set $X$ is defined in \cite{GS16} as an integer additive set-graceful labeling $f$ of a graph $G$ such that $f(V(G))\cup\{\emptyset\}$ is a topology on $X$ (see \cite{GS16}). 

The following result establishes the relation between a TIASGL and a TIASSL of $G$.

\begin{Theorem}\label{T-IASSL2}
	Every TIASGL of a given graph $G$ is a TIASSL of $G$.
\end{Theorem}
\begin{proof}
	Let $f$ be a TIASGL of a given graph $G$. Then, $f(V)\cup \{\emptyset\}$ is a topology on $X$ and $f^+(E)= \mathcal{P}(X)-\{\emptyset, \{0\}\}$. Also, we have $f(V(G))-f^+(E(G))=\{0\}$ and $f(V)-\{0\}\subseteq f^+(E(G))$. Therefore, $f(V)\cup f^+(E) = \mathcal{P}(X)-\{\emptyset\}$. Therefore, $f$ is a TIASSL of a graph $G$.
\end{proof}

The converse of this result need not be true. For, the fact $f(V)\cup f^+(E) = \mathcal{P}(X)-\{\emptyset\}$ does not necessarily imply that  $f^+(E)= \mathcal{P}(X)-\{\emptyset, \{0\}\}$. Figure \ref{fig:G-TIASSL1a} is an example of a TIASS-graph which is not a TIASG-graph with respect to the ground set $X$. Indeed, the sets $\{1,2\}$ and $\{0,2\}$ are not the set-labels of any edge of the graph and hence we can see that the subsets  $f^+(E) \ne \mathcal{P}(X)-\{\emptyset, \{0\}\}$.

\begin{figure}[h!]
\centering
\includegraphics[width=0.65\linewidth]{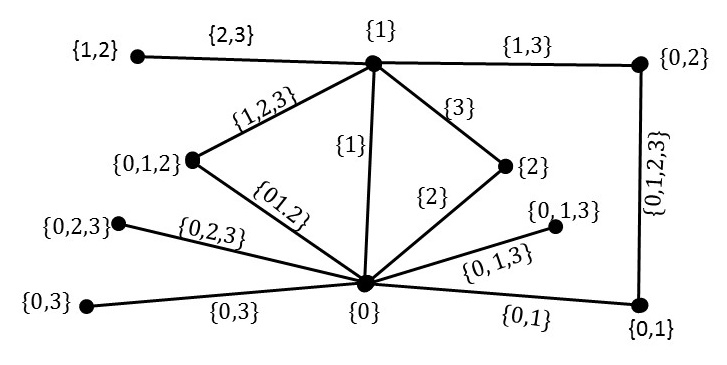}
\caption{An example of an TIASS-graph which is not a TIASG-graph}
\label{fig:G-TIASSL1a}
\end{figure}

By a \textit{graphical realisation} of a collection $\mathcal{P}(X)$ of subsets of a set $X\subseteq \mathbb{N}_0$ with respect to an IASL $f$, we mean a graph $G$ such that all non-empty elements of $\mathcal{P}(X)$ are the set-labels of some elements (edges or vertices) of $G$. 

In view of the above notions and concepts, the following theorem establishes the necessary and sufficient condition for a graph to admit a TIASSL with respect to the discrete topology of a given ground set $X$. 

\begin{Theorem}\label{T-TIASSG3}
	A graph $G$ admits a TIASSL $f$ with respect to the discrete topology of a non-empty finite set $X$ if and only if $G\cong K_{1,2^{|X|}-2}$.
\end{Theorem}
\begin{proof}
	Let $G\cong K_{1,2^{|X|}-2}$. Then, $G$ can be a graphical realisation of the collection $\mathcal{P}(X)-\emptyset$. Therefore, the corresponding IASI $f$ is a Top-IASGL of $G$. Then, by Theorem \ref{T-IASSL2}, $f$ is a TIASSL.
	
	Conversely, assume that $G$ admits a TIASSL, say $f$, with respect to the discrete topology of a non-empty finite set $X$. Therefore, $|f(V(G))|=2^{|X|}-1$ and $|f^+(E(G))|=2^{|X|}-2$. Therefore, $G$ is a tree on $2^{|X|}-1$ vertices. 	Moreover, $f(V(G))=f^+(E(G))\cup \{0\}$, which possible only when $G\cong K_{1,2^{|X|}-2}$.
\end{proof}

\noindent Invoking the above theorem, we have

\begin{Theorem}
	Let $X$ be a non-empty finite set of non-negative integers, which contains $0$. Then, for a star graph $K_{1,2^{|X|}-2}$, the existence of a Top-IASGL is equivalent to the existence of a TIASSL, with respect to the ground set $X$.
\end{Theorem}
\begin{proof}
Since $0\in X$, we have $\{0\} \in \mathcal{P}$. Now, label the central vertex of $K_{1,2^{|X|}-2}$ by $\{0\}$ and label its end vertices by the other non-empty subsets of $X$ in an injective manner. Therefore, obviously, we have $f(V(G))=\mathcal{P}(X)-\{\emptyset\}$ (and hence $f^{\ast}(G)=\mathcal{P}(X)-\{\emptyset\}$) and $f^+(E(G))=\mathcal{P}(X)-\{\emptyset, \{0\}\}$. This completes the proof.
\end{proof}

Now that the admissibility of a TIASSL by a graph with respect to the discrete topology of a given non-empty ground set $X$ has been discussed, we shall proceed to consider other topologies (containing the set $\{0\}$) of $X$ which are neither the discrete topology nor the indiscrete topology of $X$ as the set of the set-labels of vertices of $G$. In the following discussions, all the topologies we mention are neither discrete toplogies nor indiscrete toplogies of $X$.

The question that arises first about the existence of a graph $G$ that admits a TIASSL  with respect to any given finite set of non-negative integers. In the following result we establish the existence of a graph which admits a TIASL with respect to a given ground set having certain properties.

\begin{Theorem}
	Let $X$ be a finite set of non-negative integers containing $0$. Then, there exists a graph $G$ which admits an TIASSL with respect to the ground set $X$.
\end{Theorem}
\begin{proof}
	Let $X$ be a given non-empty finite set containing the element $0$. Let us begin the construction of the required graph as follows. First, choose a topology $\mathcal{T}$ on $X$ such that the non-empty sets in $\mathcal{T}$ together with their sumsets makes the power set $\mathcal{P}(X)$. 
	
	Take a vertex $v_1$ and label this vertex with the set $\{0\}$. Let $A_i\neq \{0\}$ be a subset of $X$ which is not the non-trivial sumset of any two subsets of $X$. Then, create a new vertex $v_i$ and label $v_i$ by the set $A_i$. Join this vertex to $v_1$. Repeat this process until all the subsets of $X$ which are not the non-trivial sumsets of any subsets of $X$ are used for labeling some vertices of the graph under construction. Since $f(V(G))\cup\{\emptyset\}$ must be a topology on $X$, $X$ must belong to $f(V(G))$. Hence, make an additional vertex $v$ with the set-label $X$. This vertex must also be adjacent to $v_1$.
	
	Now that all subsets of $X$ which are not sumsets of any two subsets of $X$ have been used for labeling the vertices of the graph under construction the remaining subsets of $X$ that are not used for labeling are sumsets of some two subsets of $X$, both different from $\{0\}$.  If $A_r'$ is a non-trivial sumset, then we can choose its summands to be indecomposable as non-trivial sumsets set in $\mathcal{P}(X)$ such that $A_r'= f(v_i)+f(v_j)$, for some vertices $v_i,v_j$ of the graph. Hence, draw an edge $e_r$ between $v_i$ and $v_j$ so that $e_r$ has the set-label $A_r'$. Repeat this process until all other non-empty subsets of $G$ are also the set-labels of some edges in $G$. Clearly, under this labeling we get $f^{\ast}(G)=\mathcal{P}(X)-\{\emptyset\}$. Therefore, the graph $G$ we constructed now admits a TIASSL with respect to the given ground  set $X$.
\end{proof}

Figure \ref{fig:G-TIASSL1b} depicts the existence of a TIASS-graph with respect to a given non-empty finite set of non-negative integers.

\begin{figure}[h!]
	\centering
	\includegraphics[width=0.7\linewidth]{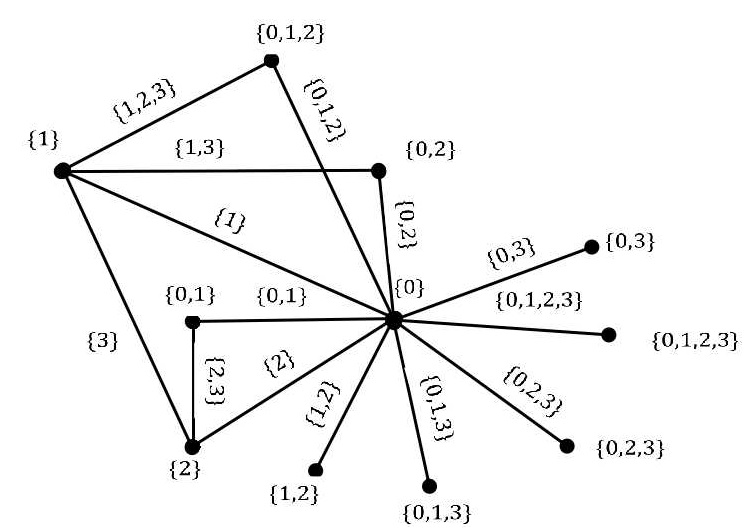}
	\caption{A TIASS-graph with respect to the ground set $X=\{0,1,2,3\}$.}
	\label{fig:G-TIASSL1b}
\end{figure}

In general, all graphs need not necessarily admit a TIASSL. It is evident from Proposition \ref{P-TIASI1}. In this context, it is relevant and interesting to determine the structural properties of TIASS-graphs. The minimum number of vertices required in a graph to admit a TIASS-graph is determined in the following result.

\begin{Proposition}\label{P-TIASSL1a}
	Let the graph $G$ admit a TIASSL $f$ with respect to a given ground set $X$ and let $\mathcal{A}$ be the set of subsets of $X$, which are not non-trivial sumsets of any subsets of $X$. If $|\mathcal{A}|=\rho$, then $G$ must have 
	\begin{enumerate}[itemsep=0mm]
		\item[(i)] at least $\rho$ vertices if $X \in \mathcal{A}$.
		\item[(ii)] at least $1+\rho$ vertices if $X \notin \mathcal{A}$.
	\end{enumerate}
\end{Proposition}
\begin{proof}
	Since $f$ is an IASI of $G$, the set-label of every edge of $G$ is the sumset of two subsets of $X$. Then, no element of $\mathcal{A}$ can be the set-label of any edge of $G$. Therefore, since $f$ is an IASSL of $G$, all the elements of $\mathcal{A}$ must be the set-labels of some vertices of $G$. 
	
	Since $f$ is a TIASL also, $X\in f(V)$. If $X\in \mathcal{A}$, then $G$ must have at least $\rho$ vertices. If $X \notin \mathcal{A}$, then $G$ must have one vertex corresponding to $X$ and one vertex corresponding to each element of $\mathcal{A}$. Therefore, in this case, $G$ must have at least $1+\rho$ vertices.
\end{proof}

In view of the above results, we note that the number of vertices adjacent to the vertex having the set-label $\{0\}$ is important in the study of TIASS-graphs. The following result describes the number of vertices that are adjacent to the vertex having the set-label $\{0\}$.

\begin{Proposition}\label{P-TIASSL1b}
	Let the graph $G$ admit a TIASSL $f$ with respect to a given ground set $X$. Then, $G$ has at least $\rho'$ vertices adjacent to the unique vertex of $G$, labeled by $\{0\}$, where $\rho'$ is the number of subsets of $X$ which are not the summands of any subsets of $X$.
\end{Proposition}
\begin{proof}
	Let $f$ be an IASSL of a given graph $G$ with respect to the ground set $X$. Let $\mathcal{B}$ be the collection of all subsets of $X$ which are not the summands of any subsets of $X$. That is, no subset of $X$ can be expressed as a sumset of any two elements in $\mathcal{B}$. Hence $|\mathcal{B}|=\rho'$. Since no element of $\mathcal{B}$ is a summand of any subset of $X$, then no vertex in $G$ which has the set-label from $\mathcal{B}$ can be adjacent only to the vertex of $G$ having the set-label $\{0\}$. Therefore, the minimum number of vertices adjacent to the vertex with set-label $\{0\}$ is $\rho'$. 
\end{proof}

Since $f$ is a TIASL, by Remark \ref{R-ITIASL1}, $G$ must have some pendant vertices. The minimum number of pendant vertices required for a graph $G$ to admit a TIASSL is determined in the following result.

\begin{Proposition}\label{P-TIASSL1c}
	Assume that the graph $G$ admits a TIASSL $f$ with respect to a given ground set $X$ and let $\mathcal{D}$ be the set of subsets of $X$ which are neither the non-trivial sumsets of two subsets of $X$ nor the summands of any subsets of $X$. If $|\mathcal{D}|=\rho''$, then $G$ must have 
	\begin{enumerate}[itemsep=0mm]
		\item[(i)] at least $\rho''$ pendant vertices if $X \in \mathcal{D}$.
		\item[(ii)] at least $1+\rho''$ pendant vertices if $X \notin \mathcal{D}$.
	\end{enumerate}
\end{Proposition}
\begin{proof}
	 By Proposition \ref{P-TIASSL1a}, no element of $\mathcal{D}$ can be the set-label of any edge of $G$. Also, since no element of $\mathcal{D}$ is a summand of any subset of $X$, no vertex in $G$ that has the set-label from of $\mathcal{D}$ can be adjacent to any other vertex in $G$ other than the one having the set-label $\{0\}$. Therefore, no element of $\mathcal{D}$ can be the set-label of an element of $G$ that is not a pendant vertex of $G$. 
	
	Clearly, $X$ is not a summand of any subset of $X$. Then, $X\in \mathcal{D}$ if and only $X$ is not the sumset of any subsets of $X$. Therefore, if $X\in \mathcal{D}$, then the minimum number of pendant vertices in $G$ is at least $\rho''$. 
	
	If $X\notin \mathcal{D}$, then $X$ is a non-trivial sumset of two subsets of $X$. But, since $f$ is a TIASL, $X\in f(V)$. Therefore, the vertex of $G$ whose set-label is $X$, can be adjacent only to the vertex of $G$ having the set-label $\{0\}$.  Therefore, in addition to the vertices with set-labels from $\mathcal{D}$, the vertex with set-label $X$ 	is also a pendant vertex. Hence, in this case, the number of pendant vertices in $G$ is at least $1+\rho''$.
\end{proof}

The number of edges in the given graph $G$ is also important in determining the admissibility of a TIASSL by $G$. Since the minimum number of vertices required in a TIASS-graph have already been determined, we can determine the maximum number of edges required in $G$. 

\vspace{0.2cm}

\noindent Invoking the above results and concepts, we propose the following theorem. 

\begin{Proposition}
	Let the graph $G$ admit a TIASSL $f$ with respect to a given ground set $X$. Let $\mathcal{A}$ be the set of subsets of $X$, which are not the sumsets of any subsets of $X$ and $\mathcal{B}$ be the collection of all subsets of $X$ which are not the summands of any subsets of $X$. Also, let $|\mathcal{A}|=\rho$ and $|\mathcal{B}|=\rho'$. Then, $G$ has at most $2^{|X|}+\rho'-\rho-1$ edges.
\end{Proposition}
\begin{proof}
	By Proposition \ref{P-TIASSL1b}, every element of $\mathcal{B}$ is the set-label of a vertex of $G$ which is adjacent to the vertex having the set-label $\{0\}$. Therefore, every element $B_i$ of $\mathcal{B}$ is the set-label of one vertex of $G$ and the edge between that vertex and the vertex with the set-label $\{0\}$. Since the set-label of every edge is the sumset of the set-labels of its end vertices, the maximum number of remaining edges in $G$ is the number of subsets of $X$ which are the sumsets two subsets of $X$. 
	
	\noindent {\em Case-1:} Let $X \notin \mathcal{A}$. Recall that $X \in \mathcal{B}$.  Then, the maximum number of edges is equal to the number of elements in the set $\mathcal{A}'$, which is the set of all (non-empty) subsets of $X$ which are non-trivial sumsets of  two subsets of $X$. We have $\mathcal{A}'= \mathcal{P}(X)-(\mathcal{A}\cup \{0\})$. Hence, maximum number of remaining edges in $G$ is $2^{|X|}-(|\mathcal{A}|+1)=2^{|X|}-\rho-1$. Therefore, in this case, the maximum number of edges in $G$ is $2^{|X|}+\rho'-\rho-1$.
	
	\noindent {\em Case-2:} Let $X \in \mathcal{A}$. Then, by Proposition \ref{P-TIASSL1b}, $G$ has at least $1+\rho'$ vertices adjacent to a unique vertex, having set-label$\{0\}$. Therefore, there are at least $1+\rho'$ edges in $G$ with set-labels from $\mathcal{B}\cup X$. The maximum number of remaining edges is equal to the number of subsets of $X$, other than $X$, which are the sumsets of any two subsets of $X$. Therefore, the maximum number of remaining edges in $G$ is $2^{|X|}-|\mathcal{A}|-2=2^{|X|}-\rho-2$. Hence,  the maximum number of edges in $G$ in this case is $1+\rho'+2^{|X|}-\rho-2=2^{|X|}+\rho'-\rho-1$.
\end{proof}

\noindent Let us summarise the above propositions in the following theorem.  

\begin{Theorem}\label{T-IASSL1}
	Let $f$ be a topological integer additive set-sequential labeling defined on a graph with respect to a given ground set $X$ containing $0$. Let $\mathcal{A}$ be the set of all subsets of $X$ which are not the non-trivial sumsets of any two subsets of $X$ and $\mathcal{B}$ be the the set of all subsets of $X$, which are non-trivial summands in some subsets of $X$. Then,
	\begin{enumerate}\itemsep0mm
		\item[(i)] the number of vertices in $G$ is at least $|\mathcal{A}|$ or $|\mathcal{A}|+1$ in accordance with whether $X$ is in $\mathcal{A}$ or not,
		\item[(ii)] the maximum number of edges in $G$ is $|\mathcal{B}|+|\mathcal{A}'|$, where $\mathcal{A}'= \mathcal{P}(X)-(\mathcal{A}\cup \{\emptyset\})$,
		\item[(iii)] the minimum degree of the vertex that has the set-label $\{0\}$ in $G$ is  $|\mathcal{B}|$,
		\item[(iv)] the minimum number of pendant vertices is $|\mathcal{A}\cap \mathcal{B}|$ or  $|\mathcal{A}\cap \mathcal{B}|+1$ in accordance with whether or not $X$ is in $\mathcal{A}$.
	\end{enumerate}
\end{Theorem}

Now, let us verify whether a graph admits a Top-IASI with respect to a given non-empty finite set $X$. 

\begin{Proposition}
	No connected graph admits a topological integer additive set-sequential indexer.
\end{Proposition}
\begin{proof}
	If possible, let $f$ be a TIASSI defined on a given connected graph $G$. Then, $f$ is a TIASL on $G$. Then, $X\in f(V)$. Then, the vertex $v$ having set-label $X$ must be a pendant vertex and it must be adjacent to the vertex $u$ having the set-label $\{0\}$. Therefore, $f^*(v)=f(v)=f^+(uv)=f^*(uv)$, a contradiction to the assumption that $f$ is an TIASSI of $G$. Hence, $f$ can not be a TIASSI of the graph $G$. 
\end{proof}

Analogous to the Theorem \ref{T-IASSL1}, we propose the characterisation of a (disconnected) graph that admits a TIASSI with respect to a given ground set $X$.

\begin{Theorem}\label{T-IASSI1}
	Let $f$ be a topological integer additive set-sequential indexer defined on a graph with respect to a given ground set $X$ containing $0$. Let $\mathcal{A}$ be the set of all subsets of $X$ which are not the non-trivial sumsets of any two subsets of $X$ and $\mathcal{B}$ be the the set of all subsets of $X$, which are non-trivial summands of any subsets of $X$. Then,
	\begin{enumerate}[itemsep=0mm]
		\item[(i)] the minimum number of vertices in $G$ is $|\mathcal{A}|$ or $|\mathcal{A}|+1$ in accordance with whether $X$ is in $\mathcal{A}$ or not,
		\item[(ii)] the minimum number of isolated vertices in $G$ is $|\mathcal{A}\cap \mathcal{B}|$ or  $|\mathcal{A}\cap \mathcal{B}|+1$ in accordance with whether $X$ is in $\mathcal{A}$ or not,
		\item[(iii)] the maximum number of edges in $G$ is $|\mathcal{A}'|$, where $\mathcal{A}'= \mathcal{P}(X)-(\mathcal{A}\cup \{\emptyset\})$.
	\end{enumerate}
\end{Theorem}
\begin{proof}
	\begin{enumerate}[itemsep=0mm]
		\item[(i)] The proof is similar to the proof of Proposition \ref{P-TIASSL1a}.
		\item[(ii)] As we have proved earlier any vertex having the set-label from $\mathcal{A}\cap \mathcal{B}$ can be adjacent only to the vertex having $\{0\}$. But, since no two elements of $G$ can have the same set-label, no element in $\mathcal{A}\cap \mathcal{B}$ can be adjacent to any other vertex of $G$. Therefore, the minimum number of isolated vertices in $G$ is $|\mathcal{A}\cap \mathcal{B}|$.
		\item[(iii)] Since all elements in $\mathcal{A}\cap \mathcal{B}$ corresponds to isolated vertices in $G$, the number of non-isolated vertices of $G$ is $|\mathcal{A}-\mathcal{B}|$. Therefore, all elements in $\mathcal{A}'$ are the sumsets of the elements in $|\mathcal{A}-\mathcal{B}|$. Therefore, the maximum number edges in $G$ is equal to $|\mathcal{A}'|=2^{|X|}-\rho-1$. 
	\end{enumerate}
\noindent This completes the proof.
\end{proof}

\section{Conclusion}

In this paper, we have discussed certain properties and characteristics of topological IASL-graphs. More properties and characteristics of TIASLs, both uniform and non-uniform, are yet to be investigated. The following are some problems which require further investigation. 

\begin{Problem}{\rm 
		Determine the smallest ground set $X$, with respect to which different graph  classes admit TIASSLs.}
\end{Problem}

\begin{Problem}{\rm
		Characterise the graphs which are TIASS-graphs, but not TIASG-graphs.}
\end{Problem}

\begin{Problem}{\rm 
		Check whether the converses of Theorem \ref{T-IASSL1} and Theorem \ref{T-IASSI1} are true.}
\end{Problem}

\begin{Problem}{\rm
		Check the admissibility of TIASSL by different operations and products of given TIASS-graphs.}
\end{Problem}

The problems of establishing the necessary and sufficient conditions for various graphs and graph classes to have certain other types IASLs are also still open. Studies of those IASLs which assign sets having specific properties, to the elements of a given graph are also noteworthy. 


\acknowledgments{Acknowledgments}

The authors would like to thank the anonymous reviewers for their insightful suggestions and critical and constructive remarks which made the overall presentation of this paper better.


\authorcontributions{Author Contributions}

All three authors have significant contribution to this paper and the final form of this paper is approved by all three authors.



\conflictofinterests{Conflicts of Interest}

The authors declare that they have no conflict of interests regarding the publication of the article.







\bibliographystyle{mdpi}
\makeatletter
\renewcommand\@biblabel[1]{#1. }
\makeatother


%


%

\end{document}